\documentclass[11pt,a4paper]{amsart}
\usepackage[colorlinks=true, pdfstartview=FitV, linkcolor=blue, 
citecolor=blue]{hyperref}

\usepackage{amssymb,bbm,amscd}
\usepackage{graphicx}
\usepackage{a4wide}

\theoremstyle{plain}
\newtheorem{theorem}{Theorem}[section]
\newtheorem{prop}[theorem]{Proposition}
\newtheorem{lemma}[theorem]{Lemma}
\newtheorem{coro}[theorem]{Corollary}

\theoremstyle{definition}
\newtheorem{remark}[theorem]{Remark}

\numberwithin{equation}{section}
 
\newcommand{\dd}{\,\mathrm{d}}
\newcommand{\ii}{\ts\mathrm{i}}
\newcommand{\ee}{\,\mathrm{e}}
\newcommand{\ts}{\hspace{0.5pt}}

\newcommand{\bigconv}{\raisebox{-3pt}{\mbox{\Huge $*$}}}
\newcommand{\medconv}{\raisebox{-2.5pt}{\mbox{\huge $*$}}}

\newcommand{\ZZ}{\mathbb{Z}}
\newcommand{\RR}{\mathbb{R}\ts}

\newcommand{\NN}{\mathbb{N}}
\newcommand{\TT}{\mathbb{T}}

\newcommand{\exend}{\hfill $\Diamond$}

\newcommand{\myfrac}[2]{\frac{\raisebox{-2pt}{$#1$}}
      {\raisebox{0.5pt}{$#2$}}}

\begin{document}

\title[A measure for Stern's sequence]{A natural probability measure 
derived from\\[2mm]  Stern's diatomic sequence}

\author{Michael Baake} 
\address{Fakult\"{a}t f\"{u}r Mathematik,
  Universit\"{a}t Bielefeld,\newline \hspace*{\parindent}Postfach
  100131, 33501 Bielefeld, Germany}
\email{mbaake@math.uni-bielefeld.de}

\author{Michael Coons}
\address{School of Mathematical and Physical Sciences,
University of Newcastle, 
\newline \hspace*{\parindent}University Drive, Callaghan NSW 2308, Australia}
\email{michael.coons@newcastle.edu.au}

\begin{abstract}
  Stern's diatomic sequence with its intrinsic repetition and
  refinement structure between consecutive powers of $2$ gives rise to
  a rather natural probability measure on the unit interval. We
  construct this measure and show that it is purely singular
  continuous, with a strictly increasing, H\"{o}lder continuous
  distribution function. Moreover, we relate this function
  with the solution of the dilation equation for Stern's diatomic
  sequence.
\end{abstract}

\maketitle
\thispagestyle{empty}

\section{Introduction}

Stern's diatomic sequence $\bigl( s(n) \bigr)_{n\in\NN_0}$, also known
as the Stern--Brocot sequence, is defined by $s(0)=0$, $s(1)=1$
together with the recursions
\begin{equation}\label{eq:def-Stern}
   s(2n) \, = \, s(n)
   \quad \text{and} \quad
   s(2n+1) \, = \, s(n) + s(n+1)
\end{equation}
for $n \in \NN$. This well-studied sequence has fascinating
properties; see entry \textsf{A{\ts}002487} of \cite{oeis} for a
concise summary with many references and links. The initial values and
recursions in Eq.~\eqref{eq:def-Stern} allow one to determine the value
$s(n)$ based on the binary expansion of $n$. In particular, if
\begin{equation}\label{eq:linrec} 
   {S}^{}_0 \, = \,\left(\begin{matrix} 1& 0\\ 
    1& 1\end{matrix}\right),\quad 
   {S}^{}_1 \, = \, \left(\begin{matrix} 1& 1\\ 
   0& 1\end{matrix}\right),\quad {v} \, = \,
   \left(\begin{matrix}1 \\ 0\end{matrix}\right),
\end{equation} 
and if $(n)^{}_2=b^{}_k b^{}_{k-1} \cdots b^{}_1 b^{}_0$ is the binary
expansion of $n$, one has
\begin{equation}\label{eq:form}
  s(n) \, = \, {v}^{T} {S}_{b^{}_k} {S}_{b^{}_{k-1}} \cdots
     {S}_{b^{}_1} {S}_{b^{}_0} {v} \ts .
\end{equation}
Sequences with a linear representation as provided by
Eqs.~\eqref{eq:linrec} and \eqref{eq:form} are called
\mbox{$b\ts$-regular} sequences, where $b$ is the base ($b=2$ for
Stern's diatomic sequence). Regular sequences were introduced by
Allouche and Shallit \cite{AS1992} as a mathematical generalisation of
sequences that are generated by deterministic finite automata, such as
the Thue--Morse sequence.

Here, we reconsider the self-similarity type property of Stern's
diatomic sequence, which manifests itself in the fact that the
sequence, in the range from $2^n$ to $2^{n+1}$, can be seen as a
stretched and interlaced version of what it is between $2^{n-1}$ and
$2^n$. In particular, as follows from a simple induction argument, one
has the well-known summation relation
\begin{equation}\label{eq:sum}
     \sum_{m=2^n}^{2^{n+1} - 1} \!\! s(m) \, = \, 3^n ,
\end{equation}
which holds for all $n\in\NN_0$. Therefore, if we define
\begin{equation}\label{eq:def-meas}
   \mu^{}_{n} \, := \, 3^{-n} \sum_{m=0}^{2^n - 1}
   s(2^n + m) \, \delta^{}_{m / 2^{n}} \ts ,
\end{equation}
where $\delta_x$ denotes the unit Dirac measure at $x$, we can view
$(\mu^{}_{n})^{}_{n\in\NN_0}$ as a sequence of probability measures on
the $1$-torus, the latter written as $\TT=[0,1)$ with addition modulo
$1$. Here, we have simply re-interpreted the values of the Stern
sequence between $2^n$ and $2^{n+1}-1$ as weights of a pure point
probability measure on $\TT$ with
$\mathrm{supp} (\mu^{}_n ) = \big\{ \frac{m}{2^n} : 0 \leqslant m <
2^n \big\}$.

In the remainder of this article, we will study the sequence
$(\mu_{n})^{}_{n\in\NN_{0}}$ and its limit, as well as various
properties of the latter and how they relate to other known results on
Stern's sequence. Our approach is motivated by the similarity of
Eq.~\eqref{eq:def-Stern} with the recursion relation for the
Fourier--Bohr coefficients of the classic Thue{\ts}--Morse measure and
some of its generalisations; see \cite[Sec.~10.1]{TAO} as well as
\cite{BGG} and references therein for background. The idea of studying
the asymptotic behaviour of self-similar sequences by defining a
related measure via renormalisation and Bernoulli convolution has a
rich history, dating back to at least the late 1930s to two papers of
Erd\H{o}s \cite{E1939,E1940}; see \cite{PSS} for a comprehensive
survey.

\section{The probability measure}

Each $\mu^{}_n$ is a probability measure on $\TT$, which in particular
implies that it is Fourier (or Fourier--Stieltjes) transformable,
where
\[
    k \, \mapsto \, \widehat{\mu^{}_n} (k) 
    \, := \int_{\TT}
    \ee^{-2 \pi \ii k x} \dd \mu^{}_{n} (x) 
\]
defines a continuous function on the dual group $\widehat{\TT} = \ZZ$;
see \cite[Sec.~4.4]{RSt} for background.

\begin{remark}\label{rem:period}
  It is sometimes useful to `periodise' the measure $\mu^{}_{n}$
  to $\nu^{}_{n} := \mu^{}_{n} * \delta^{}_{\ZZ}$ and interpret it
  as a translation bounded measure on $\RR$. Its Fourier transform
  is still well defined, via the Poisson summation formula
  $\widehat{\delta^{}_{\ZZ}} = \delta^{}_{\ZZ}$ and the convolution
  theorem; compare \cite[Prop.~8.5 and Sec.~9.2]{TAO}. It then reads
\[
    \widehat{\nu^{}_{n}} \, = \,  \widehat{\mu^{}_n}
    \, \delta^{}_{\ZZ} \, = \sum_{x\in\ZZ}
    \widehat{\mu^{}_n} (x) \, \delta^{}_{x} \ts ,
\]
where $ \widehat{\mu^{}_n}$ now defines a (continuous) function on
$\RR$. The values of $\widehat{\mu^{}_n}$ in the complement
of $\ZZ$ are irrelevant, but still useful; compare
\cite[Sec.~9.2.4]{TAO} for a general interpretation of this
phenomenon.  \exend
\end{remark}

Let us analyse the functions $\widehat{\mu^{}_n}$.
Clearly, one has $\widehat{\mu^{}_{0}}\equiv 1$ and 
\[
    \widehat{\mu^{}_1} (k) \, = \,
    \tfrac{1}{3} \bigl( 1 + 2 \cos (\pi k) \bigr)
    \, = \, \begin{cases} 1, & \text{$k$ even}, \\
         - \frac{1}{3}, & \text{$k$ odd},
    \end{cases}
\]
where $\widehat{\delta_x} (k) = \ee^{-2 \pi \ii kx}$ was used in the
calculation.  More generally, by induction on the basis of
Eqs.~\eqref{eq:def-meas} and \eqref{eq:def-Stern}, one finds
\begin{equation}\label{eq:fin-prod}
    \widehat{\mu^{}_{n}} (k) \, = \prod_{m=1}^{n}
    \tfrac{1}{3} \! \left( 1 + 2 \cos \bigl(
    \tfrac{2 \pi k}{2^m} \bigr) \right)
\end{equation}
for $n\in\NN_0$ and $k\in\ZZ$, where the empty product is defined to
be $1$ as usual. Since
\[
    \tfrac{1}{3} \bigl( 1 + 2 \cos (x) \bigr)
    \, = \, 1 - \tfrac{1}{3} x^2 + O(x^4)
\]
as $\lvert x \rvert \searrow 0$, one can apply standard arguments to
show that, for any fixed $k$, the sequence
$\bigl( \widehat{\mu^{}_n} (k) \bigr)_{n\in\NN_0}$ converges. In fact,
one has compact convergence, both for $k\in\ZZ$ and for $k\in\RR$.
The latter viewpoint is useful in the context of
Remark~\ref{rem:period}, and will be vital later.

Let us now formulate some consequences, where the location
$x$ of the Dirac measure $\delta_x$ is always understood to be an
element of $\TT$, hence taken modulo $1$. This is important to
give the correct meaning to the convolution identity $\delta_x * 
\delta_y =\delta_{x+y}$ on $\TT$.

\begin{prop}\label{prop:gen}
  The sequence\/ $(\mu^{}_n)^{}_{n\in\NN_0}$ of probability measures
  on\/ $\TT$ converges weakly to a probability measure\/ $\mu$.  In
  particular, one has\/ $\mu^{}_0 =\delta^{}_0$ and\/
  $\mu^{}_{n} = \medconv_{m=1}^{n} \, \tfrac{1}{3} \bigl( \delta^{}_0 +
  \delta^{}_{2^{-m}} + \delta^{}_{-2^{-m}} \bigr)$
  for\/ $n\geqslant 1$. The weak limit as\/ $n\to\infty$ is given by
  the convergent infinite convolution product\/
\[
   \mu \,  =  \underset{m\geqslant 1}{\bigconv} 
   \, \tfrac{1}{3} \bigl(
  \delta^{}_0 + \delta^{}_{2^{-m}} + \delta^{}_{-2^{-m}} \bigr).
\]
Its Fourier transform\/ $\widehat{\mu}$ is given by\/
$\widehat{\mu} (k) = \prod_{m\geqslant 1} \frac{1}{3} \bigl( 1 + 2
\cos(2 \pi k/2^m) \bigr)$
for\/ $k\in\ZZ$. Moreover, this infinite product is also well-defined
on\/ $\RR$, where it converges compactly.
\end{prop}

\begin{proof}
  Due to the convergence of the sequences
  $\bigl( \widehat{\mu^{}_n} (k) \bigr)_{n\in\NN_0}$,
  the first claim can be seen as a consequence of Levy's continuity
  theorem \cite[Thm.~3.14]{BF}. The explicit formula for
  $\mu^{}_n$ follows from Eq.~\eqref{eq:fin-prod} with a simple
  calculation via the inverse of the convolution theorem.

  The representation of $\mu$ is clear, with weak convergence,
  as is the formula for $\widehat{\mu} (k)$ with compact convergence
  of the infinite product as mentioned above.
\end{proof}

Since $\mu$ is a probability measure on $\TT$, Bochner's theorem
\cite[Thm.~3.12]{BF} implies that $k \mapsto \widehat{\mu} (k)$ 
defines a (continuous) positive definite function on $\ZZ$. 
In particular, one has
\begin{equation}\label{eq:symm}
    \widehat{\mu} (-k) \, = \, \overline{\widehat{\mu} (k)}
    \, = \, \widehat{\mu} (k)
\end{equation}
for all $k\in\ZZ$. Here, $\widehat{\mu}$ is real (which gives the second
equality) as a consequence of the invariance of $\mu$ on $\TT$ under
the reflection $x\mapsto -x$, again taken modulo $1$, while the
normalisation of $\mu$ corresponds to $\widehat{\mu} (0) = 1$. The
symmetry relation also implies that
\[
     \widehat{\mu} (k) \, = \int_{\TT} \ee^{2 \pi \ii k x} \dd \mu (x)
     \, =  \int_{\TT} \cos (2 \pi k x) \dd \mu (x)
\]
holds for all $k\in\ZZ$.

The representation of $\mu$ as an infinite convolution
product of pure point measures allows
us to use a result by Jessen and Wintner \cite[Thm.~35]{JW} which
tells us that the spectral type of $\mu$ is pure. By the
general Lebesgue decomposition theorem, this means that 
 $\mu$ is
either  a pure point measure, or purely singular continuous, or
purely absolutely continuous --- but not a mixture. Its remains
to determine the type, for which
we need a scaling property of the Fourier--Bohr coefficients
$\widehat{\mu} (k)$.

\begin{lemma}\label{lem:FB}
  For all\/ $k\in\RR$, the coefficients\/ $\widehat{\mu} (k)$ from
  Proposition~$\ref{prop:gen}$ satisfy 
\[
   \widehat{\mu} (2k) \, = \,
   \tfrac{1}{3} \bigl( 1 + 2 \cos(2 \pi k)\bigr)
     \, \widehat{\mu} (k) \ts.
\]
In particular, $\widehat{\mu} (2k) = \widehat{\mu} (k)$ for all\/
$k\in\ZZ$.
\end{lemma}

\begin{proof}
  Since the infinite product representation of $\widehat{\mu}(k)$ is
  absolutely converging by standard arguments, one may simply
  calculate
\[
   \widehat{\mu} (2k) \, = \prod_{m\geqslant 1}
   \tfrac{1}{3} \bigl( 1 + 2 \cos(2 \pi k/2^{m-1}) \bigr)
   \, = \, \tfrac{1}{3} \bigl( 1 + 2 \cos(2 \pi k)\bigr)
   \, \widehat{\mu} (k) \ts ,
\]
which obviously implies both relations.
\end{proof}

The relationship for real $k$ in Lemma~\ref{lem:FB} has some immediate
implications. Let us first note that, for each positive integer $N$,
one has
\begin{equation}\label{eq:2Nk}
   \widehat{\mu} (2^N k) \, = \, \widehat{\mu} (k)\, 
   \prod_{m=1}^{N}  \tfrac{1}{3} 
   \bigl( 1 + 2 \cos(2^m \pi k)\bigr).
\end{equation}
For the proof of our next result, we will require information about
$|\widehat{\mu} (k)|^2$. Since the product on the right hand side of
Eq.~\eqref{eq:2Nk} is symmetric around $\frac{1}{2}$ in $[0,1]$ for
each $N\in\NN$, we can profit from relating values of
$\widehat{\mu}(k)$ with $ 0 \leqslant k \leqslant 2^N$ to values of
$\widehat{\mu}(\kappa)$ with $\kappa \in [0,1]$; see
Figure~\ref{fig:muk01} for an illustration of
$\lvert\widehat{\mu} (\kappa)\rvert$. For larger values of $\kappa$,
the function values $\widehat{\mu} (\kappa)$ are generally
small, with (bounded) negative excursions at powers of $2$.

Let us now observe that, for
$\kappa\in \bigl[ 0, \frac{1}{2} \bigr]$, we clearly have
\begin{equation}\label{eq:2Nkratio-1}
  \frac{\widehat{\mu} \bigl(2^N (1-\kappa )\bigr)}
       {\widehat{\mu} \bigl(2^N \kappa \bigr)}
  \, = \, \frac{\widehat{\mu} (1-\kappa )}
          {\widehat{\mu} (\kappa )} \ts .
\end{equation}
Since $|\widehat{\mu} (\kappa )| \geqslant |\widehat{\mu} 
(1-\kappa )|$ on this interval, we obtain the estimate
\begin{equation}\label{eq:2Nkratio}
   \big\lvert \widehat{\mu} \bigl( 2^N(1- \kappa ) \bigr) 
   \big\rvert  \, \leqslant \:
   \big\lvert \widehat{\mu} \bigl( 2^N \kappa \bigr) 
   \big\rvert \ts ,
\end{equation}
again for $\kappa \in \bigl[ 0, \frac{1}{2} \bigr]$, which implies
$\big\lvert \widehat{\mu} (2^N \! - k)\big\rvert \leqslant \big\lvert
\widehat{\mu} (k)\big\rvert$ for $0\leqslant k \leqslant 2^{N-1}$.

\begin{figure}
  \includegraphics[width=0.7\textwidth]{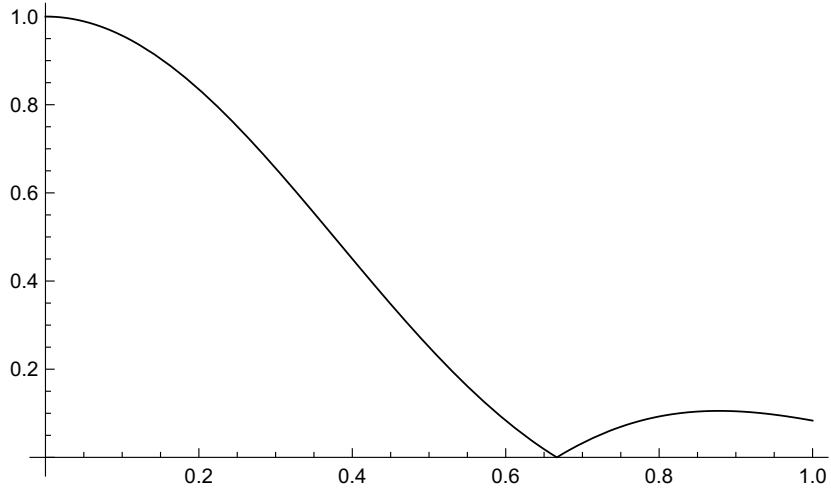}
\caption{The function $|\widehat{\mu} (\kappa)|$ for
  $0 \leqslant \kappa \leqslant 1$.\label{fig:muk01}}
\end{figure}

\begin{theorem}\label{thm:cont}
  The probability measure\/ $\mu$ from Proposition~\emph{$\ref{prop:gen}$} 
  is purely singular continuous.
\end{theorem}

\begin{proof}
  Clearly, $\mu\ne 0$, and the spectral type of $\mu$ is pure by
  \cite[Thm.~35]{JW}. Thus, we can prove the result by showing that
  $\mu$ is neither absolutely continuous nor pure point.

  Since $0 \ne \widehat{\mu} (1) = \widehat{\mu} (-1) \approx -
  0.083{\ts}432$
  and $\widehat{\mu} (2k) = \widehat{\mu} (k)$ for all $k\in\ZZ$ by
  Lemma~\ref{lem:FB}{\ts}, the Fourier coefficients cannot decay as
  $\lvert k \rvert \to \infty$. Consequently, $\mu$ cannot be absolutely
  continuous by the Riemann--Lebesgue lemma; compare
  \cite[Thm.~4.4.3]{RSt}.

  To rule out a pure point nature of $\mu$, we employ Wiener's
  criterion; see \cite[Prop.~8.9]{TAO}. Because
  $\widehat{\mu} (k)=\widehat{\mu} (-k)$ for all $k$, the theorem will
  follow if $\sum_{k\leqslant x}|\widehat{\mu} (k)|^2 = o(x)$ as
  $x\to\infty$, where here and below the summation for $k$ starts at
  $0$ unless specified otherwise. Moreover, when $x\in[2^N,2^{N+1}]$,
  one has the estimate
\[
  \myfrac{1}{x}\sum_{k\leqslant x}|\widehat{\mu} (k)|^2
  \, \leqslant \,
  \myfrac{1}{2^N}\sum_{k=0}^{2^{N+1}}|\widehat{\mu} (k)|^2
  \, \leqslant \,
  \myfrac{2}{2^{N+1}}\sum_{k=0}^{2^{N+1}}|\widehat{\mu} (k)|^2,
\]
which implies that it suffices to show
$\sum_{k\leqslant 2^{N+1}}|\widehat{\mu} (k)|^2=o \bigl( 2^{N} \bigr)$.

To this end, note that using Eq.~\eqref{eq:2Nkratio-1} we have
\[
 \sum_{k\leqslant 2^N}|\widehat{\mu} (k)|^2 
  \,  =\sum_{k\leqslant \tfrac{3}{5}\ts 2^N}|\widehat{\mu} (k)|^2
   \, + \sum_{\tfrac{3}{5}\ts  2^N< k\leqslant 2^N}|\widehat{\mu} (k)|^2
   \, \leqslant \sum_{k\leqslant \tfrac{3}{5}\ts  2^N}
   |\widehat{\mu} (k)|^2 \, + \sum_{k\leqslant \tfrac{2}{5}\ts  2^N}
   | r(k) \, \widehat{\mu} (k) |^2
\]
where, due to $0 \leqslant k \leqslant \frac{2}{5} \ts\ts 2^N$,
\[
\begin{split}
  \lvert r(k) \rvert \, & = \, \Bigl| \frac{\widehat{\mu} 
  (2^N \! - k)} {\widehat{\mu} (k)} \Bigr| \, = \, 
  \Bigl| \frac{\widehat{\mu} (1- k/2^N )}
  {\widehat{\mu} (k/2^N )} \Bigr| \, \leqslant \,
  \frac{\max_{\ts 0 \leqslant \ell \leqslant \frac{2}{5} \ts 2^N} 
         \big|\widehat{\mu} (1-\ell/2^N)\big|}
         {\min_{\ts 0\leqslant \ell \leqslant \frac{2}{5}\ts  2^N} 
         \big|\widehat{\mu} (\ell/2^N )\big|} \\[1mm]
  & \leqslant \, 
  \frac{\max_{\kappa \in [\frac{3}{5} , 1]} 
         \big|\widehat{\mu} (\kappa )\big|}
         {\ts \min_{ \kappa \in [\ts 0, \frac{2}{5}]} 
         \big|\widehat{\mu} (\kappa)\big|}
  \, = \, \frac{ \big\lvert \widehat{\mu} (t) \big\rvert }
        {\big\lvert \widehat{\mu} 
        (\tfrac{2}{5}\bigr) \big\rvert} \, = \,
  \frac{0.105{\ts}423{\ts}890\ldots}{0.450{\ts}342{\ts}617\ldots}
  \, < \, \myfrac{1}{4} \ts ,
\end{split}
\]
with $t=0.877{\ts}996{\ts}139\ldots$ being the position of the unique
(relative) maximum of $\lvert \widehat{\mu} \rvert$ in the interval
$\bigl[ \frac{3}{5}, 1 \bigr]$; compare Figure~\ref{fig:muk01}. So, we
get
\begin{align*}
   \sum_{k\leqslant 2^N}|\widehat{\mu} (k)|^2 \;
    & < \sum_{k\leqslant \frac{3}{5}\ts  2^N}
      |\widehat{\mu} (k)|^2 \, + \myfrac{1}{16}
      \sum_{k\leqslant \frac{2}{5}\ts  2^N}
      |\widehat{\mu} (k)|^2\\[1mm]
    & = \: \myfrac{17}{16}\sum_{k\leqslant 2^{N-1}}
    |\widehat{\mu} (k)|^2 \, +
    \sum_{2^{N-1}<k\leqslant \frac{3}{5}\ts  2^N}
    |\widehat{\mu} (k)|^2 \, - \, \myfrac{1}{16}
    \sum_{\frac{2}{5}\ts  2^N<k\leqslant 2^{N-1}}|\widehat{\mu} (k)|^2.
\end{align*}
To obtain an upper estimate of the last two sums in the previous line,
we write $2^N = 2 \cdot 2^{N-1}$ and use Eq.~\eqref{eq:2Nkratio} to get
\begin{align*}
  \sum_{2^{N-1}<k\leqslant \frac{6}{5}\ts 2^{N-1}}
   |\widehat{\mu} (k)|^2 \, & - \, \myfrac{1}{16}
   \sum_{\frac{4}{5}\ts 2^{N-1}<k\leqslant 2^{N-1}}|\widehat{\mu} (k)|^2 \\[1mm]
  &\leqslant \sum_{\frac{4}{5}\ts 2^{N-1}<k\leqslant 2^{N-1}}
     |\widehat{\mu} (k)|^2 \, - \, \myfrac{1}{16}
    \sum_{\frac{4}{5}\ts 2^{N-1}<k\leqslant 2^{N-1}}|\widehat{\mu} (k)|^2 \\[1mm]
  &= \, \myfrac{15}{16}\sum_{\frac{4}{5}\ts 2^{N-1}<k\leqslant 2^{N-1}}
     |\widehat{\mu} (k)|^2\\[1mm]
  &\leqslant \, \myfrac{15}{16}\sum_{k\leqslant 2^{N-1}}
     |\widehat{\mu} (k)|^2 \, - \, \myfrac{15}{16}
    \sum_{k\leqslant 2^{N-2}}|\widehat{\mu} (k)|^2.
\end{align*}
With $\varSigma^{}_N := \tfrac{1}{2^N}\sum_{k\leqslant 2^N}|\widehat{\mu}
(k)|^2$, inserting the last estimate into the previous one gives
\begin{equation}\label{eq:sublinear}
  \varSigma^{}_N \, < \, \varSigma^{}_{N-1}
   -\myfrac{15}{64} \ts \varSigma^{}_{N-2} \ts .
\end{equation}
Substituting the recursive inequality \eqref{eq:sublinear} for
$\varSigma^{}_{N-1}$ with $N\geqslant 2$ results in
\begin{equation*}
   \varSigma^{}_N \, < \, \myfrac{49}{64}\ts 
   \varSigma^{}_{N-2} - \myfrac{15}{64}\ts \varSigma^{}_{N-3}
   \,\leqslant\, \myfrac{49}{64}\ts \varSigma^{}_{N-2} \ts ,
\end{equation*} 
where $\varSigma^{}_{-1}:=0$. Consequently, for $N\geqslant 2$, we
obtain
\[
   0 \, \leqslant \, \varSigma^{}_N \, < \, 
   \Bigl(\myfrac{7}{8}\Bigr)^{N-1} 
   \max \big\{ \varSigma^{}_0, \varSigma^{}_1 \big\} 
   \, \xrightarrow{\, N\to\infty\,} \, 0 \ts ,
\]
which proves the absence of pure point components for $\mu$,
and completes our argument.
\end{proof}

\section{The distribution function}

Here, we are in a situation that is somewhat similar to that of the
singular continuous diffraction measures known from the spectral
theory of certain substitution systems; compare \cite{squiral} and
references therein. First of all, the distribution function
\begin{equation}\label{eq:distmu}
    F(x) \, := \, \mu \bigl( [0,x] \bigr)
\end{equation}
with $x\in\TT$, where $F(1):= 1$, defines a continuous function
that is monotonically increasing on the unit interval. It is
illustrated in Figure~\ref{fig:dist}.

\begin{figure}
  \includegraphics[width=0.7\textwidth]{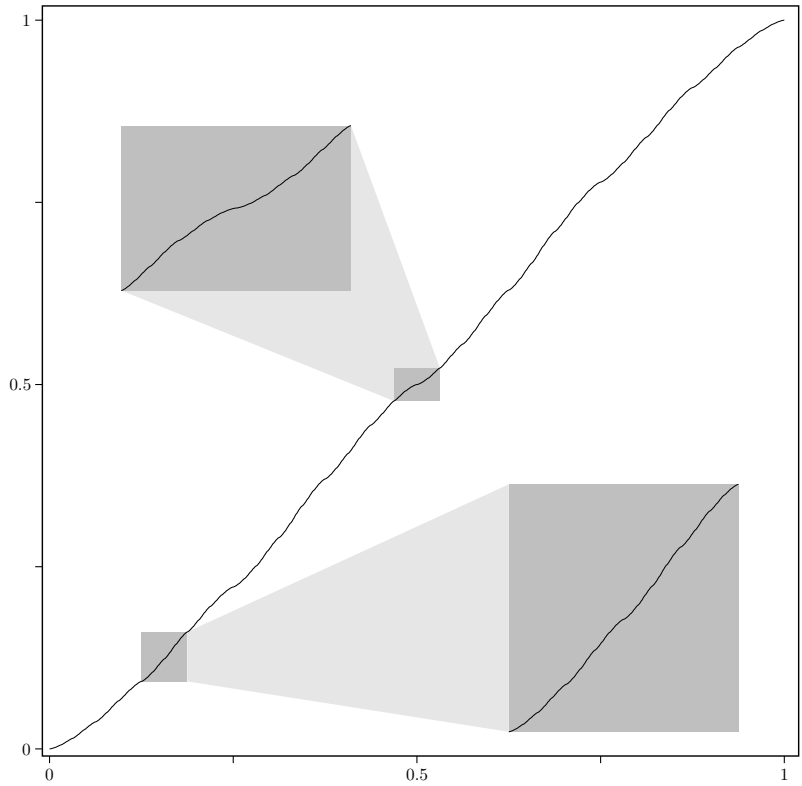}
  \caption{The distribution function $F$ of the purely singular
    continuous probability measure $\mu$ derived from Stern's
    diatomic sequence.\label{fig:dist}}
\end{figure}

It is clear from Theorem~\ref{thm:cont}
that $F$ is a non-decreasing, continuous function. Next,
we will show that $F$ is strictly increasing. To do this, we need some
specific asymptotics on the summatory function of Stern's sequence. In
what follows, for a positive real number $y$, we use
$\lfloor y\rfloor$ and $\langle y \rangle$ 
to denote the integer and the fractional part\footnote{We use this
  version for the fractional part because the more common notation,
  $\{ y \}$, represents a singleton set in our measure-theoretic 
  arguments.} of $y$, respectively. In particular,
$y=\lfloor y\rfloor+\langle y \rangle$. Also, we write $\log_2 (y)$
for the base-$2$ logarithm of $y$ and let
$\tau=\bigl(1+\sqrt{5}\, \bigr)/2$ denote the golden ratio.

\begin{prop}\label{prop:s} 
  If\/ $\bigl(s(n)\bigr)_{n\geqslant 0}$ is Stern's diatomic sequence, its
  summatory function satisfies
\[
   \sum_{n\leqslant x} s(n) \, = \, 3^{\lfloor \log_2 (x)\rfloor+1} \,
    f^{}_{0} \big( 2^{\langle \log_2 (x) \rangle - 1} \big)\ts + 
    O\big(x^{\log_2 (\tau) + \ts \varepsilon}\big),
\]
where the function\/ $f^{}_{0} $ is H\"older continuous with
exponent\/ $\log_2(3/\tau)$. Moreover, $f^{}_0 (t)$ is the first
coordinate of the column vector\/
${f}(t)= \bigl(f^{}_0(t) , f^{}_1(t)\bigr)^T$ that is the unique
solution of the dilation equation
\[
   {f}(t) \, = \, \myfrac{1}{3}\ts \bigl(
   {S}^{}_{0} \ts {f}(2t) + {S}^{}_{1} \ts {f}(2t-1)\bigr) ,
\]
with the conditions\/ $f^{}_0 (t) = f^{}_1 (t) = 0$ for\/
$t\leqslant 0$ and\/ $f^{}_0 (t) = f^{}_1 (t) = \frac{1}{2}$ for\/
$t\geqslant 1$.
\end{prop}

\begin{proof}[Sketch of proof]
  This result follows from a method of Dumas. In particular, it
  follows from an application of \cite[Thm.~3]{D2013}, when using the
  linear representation of the Stern sequence \eqref{eq:form} along
  with the facts that the set $\{{S}^{}_0 , {S}^{}_1\}$ satisfies
  the finiteness property with joint spectral radius equal to the
  golden ratio (see \cite{CIJFCS,CT2014}), that
  ${Q} := {S}^{}_0 + {S}^{}_1$ has eigenvalues $3$ and $1$ with
  Jordan basis
  ${v}^{}_{3} = \bigl(\frac{1}{2}, \frac{1}{2}\bigr)^T$ and
  ${v}^{}_{1} = \bigl( \frac{1}{2}, -\frac{1}{2} \bigr)^T$, and
  that ${v} = {v}^{}_{3} + {v}^{}_{1}$.
\end{proof}

\begin{figure}[htp]
   \includegraphics[width=0.7\textwidth]{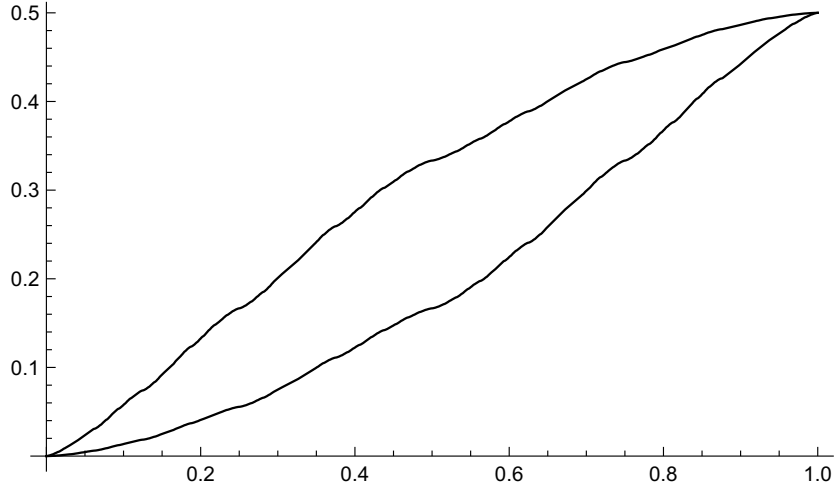}
\caption{The functions $f^{}_{0}$ (lower curve) and
  $f^{}_{1}$ (upper curve).}
\label{Sternperiod}
\end{figure}

\begin{theorem}\label{thm:strict} 
  The distribution function\/ $F$ from Eq.~\eqref{eq:distmu} is strictly
  increasing.
\end{theorem}

\begin{proof} 
  Let $x\in [0,1)$ and $\varepsilon>0$ with $x+\varepsilon\leqslant 1$
  be arbitrary. Since $F$ is continuous and non-decreasing, we have
  that
\[
   F(x+\varepsilon)-F(x) \, = \, \mu([0,x+\varepsilon])-\mu([0,x])
   \, = \, \mu([x,x+\varepsilon]) \, \geqslant \, 0 \ts .
\]
Thus, to prove the theorem, it is enough to show that
$\mu([x,x+\varepsilon]) > 0$. Note also that, since the dyadic
rationals are dense in $[0,1]$, there exist positive integers $m$ and
$k$ so that $[2m/2^k, (2m+1)/2^k]\subseteq [x,x+\varepsilon]$. Thus,
to show that $\mu([x,x+\varepsilon]) > 0$, it suffices to show that
$\mu([2m/2^k, (2m+1)/2^k]) > 0$.

To this end, observe that, for any real number $y\in[2^n,2^{n+1})$, we
have $\lfloor \log_2 (y) \rfloor =n$. Thus, combining
Eq.~\eqref{eq:def-meas} with Proposition~\ref{prop:s} gives
\begin{align*}
  \mu_n \bigl( [2m/2^k, (2m+1)/2^k] \bigr) & = \, 
   3^{-n} \sum_{\ell=2^n+2^n(2m/2^k)}^{2^n +2^n((2m+1)/2^k)} s(\ell)\\[2mm]
   & = \, 3 \ts f^{}_{0} 
     \big(2^{\langle \log_2 (2^n +2^n((2m+1)/2^k))\rangle-1}\big) -
     3\ts f^{}_{0}\big(2^{\langle \log_2 (2^n+2^n(2m/2^k))\rangle-1}\big)\\
   &\qquad\qquad + O_n \left( 3^{-n} \bigl(
     2^n +2^n((2m+1)/2^k) \bigr)^{\log_2 (\tau)
          +\varepsilon} \right) ,
\end{align*} 
where we have used the notation $O_n$ to indicate the dependence on
$n$. Since $m$ and $k$ are fixed, one has
\[
    O_n \big((2^n +2^n((2m+1)/2^k))^{\log_2 (\tau)
    +\varepsilon}\, 3^{-n}\big) \, = \, O_n\big(2^{n(\log_2 (\tau)
    +\varepsilon)}\, 3^{-n}\big) \, = \, o_n(1) \ts , \quad
    \text{as $n \to\infty$,}
\]
and, for any $y\in(0,1)$,
\[
  \langle \log_2 (2^n+2^ny)\rangle \, = \,
    \langle n+\log_2 (1+y)\rangle
   \, = \, \log_2 (1+y) \ts .
\]
Continuing the above then gives 
\[
   \mu_n \bigl( [2m/2^k, (2m+1)/2^k ] \bigr) \, = \, 3\ts
    f^{}_{0} \left(\frac{((2m+1)/2^k)+1}{2}\right) - 
     3\ts f^{}_{0} \left(\frac{(2m/2^k)+1}{2}\right) + o_n(1) \ts .
\]
Taking the limit as $n$ goes to infinity gives
\begin{equation}\label{eq:fdiffm}
   \mu \bigl( [2m/2^k, (2m+1)/2^k ] \bigr) \, = \,
   3\ts f^{}_{0} \Bigl(\myfrac{2^k+2m+1}{2^{k+1}}\Bigr)-
   3\ts f^{}_{0} \Bigl(\myfrac{2^k+2m}{2^{k+1}}\Bigr).
\end{equation}
We can now use the dilation equation for ${f}$ to determine the value
of the right-hand side of \eqref{eq:fdiffm} based on the binary
expansion of the numerators of the dyadic rationals involved. 
In particular, if $b_i\in\{0,1\}$ for $i \in \{0,1,\ldots, k\}$,  we have
\[
\begin{split}
  {f}  & \left( \frac{b^{}_{k} 2^k + b^{}_{k-1} 2^{k-1} +
        \cdots + b^{}_1 \ts 2 + b^{}_0}{2^{k+1}}\right)  \\[3mm]
   & \quad\quad = \, \left\{ \begin{array}{ll}
        \dfrac{1}{6} \left(\begin{matrix}
        1 & 0\\ 1& 1\end{matrix}\right) \left(\begin{matrix} 
        1 \\ 1 \end{matrix}\right)
        + \dfrac{1}{3} \left(\begin{matrix} 
            1 & 1\\ 0& 1\end{matrix}\right)
        {f}\left( \dfrac{b^{}_{k-1} 2^{k-1}+\cdots+b^{}_1 \ts 2 + b^{} _0}
             {2^{k}}\right), & \mbox{if $b^{}_k=1$}  \\[5mm]
    \dfrac{1}{3}\left(\begin{matrix} 1 & 0\\ 1& 1\end{matrix}\right)
    {f}\left(\dfrac{b^{}_{k-1} 2^{k-1} + \cdots + b^{}_1 \ts 2 + b^{}_0}
     {2^{k}}\right), & \mbox{if $b^{}_k=0$} \end{array}\right. \\[3mm]
    &\quad\quad=\dfrac{b^{}_k}{6} \left(\begin{matrix} 
     1 & 0\\ 1& 1\end{matrix}\right)\left(\begin{matrix}
      1 \\ 1 \end{matrix}\right)+\left[\dfrac{b^{}_k}{3}
      \left(\begin{matrix} 1 & 1\\ 0& 1\end{matrix}\right)
      +\dfrac{1-b^{}_k}{3}\left(\begin{matrix} 
       1 & 0\\ 1& 1\end{matrix}\right)\right] {f}
      \left(\dfrac{b^{}_{k-1} 2^{k-1} + \cdots + b^{}_1
        \ts 2 + b^{}_0}{2^{k}}\right)\\[3mm]
     &\quad\quad=\dfrac{b^{}_k}{6}\left(\begin{matrix} 
     1 & 0\\ 1& 1\end{matrix}\right)\left(\begin{matrix} 1 \\
      1 \end{matrix}\right)+\dfrac{1}{3}\left(\begin{matrix} 
      1 & 1\\ 0& 1\end{matrix}\right)^{b^{}_k}\left(\begin{matrix} 
     1 & 0\\ 1& 1\end{matrix}\right)^{1-b^{}_k} {f}
     \left(\dfrac{b^{}_{k-1} 2^{k-1} + \cdots + b^{}_1 
       \ts 2 + b^{}_0}{2^{k}}\right).
\end{split}
\]
Note next that
\[
   \left(\begin{matrix} 1 & 1\\ 0&
    1\end{matrix}\right)^{b^{}_k}\left(\begin{matrix} 1 & 0\\ 1&
    1\end{matrix}\right)^{1-b^{}_k}=\, {S}^{}_{b^{}_k} \ts ,
\]
where ${S}^{}_{b^{}_k}$ is as given in the linear representation of
$s(n)$. Thus setting
\[
   {u}^{}_{b_k} \, := \, \dfrac{b^{}_k}{2}
    \left(\begin{matrix} 1 & 0\\ 1& 1\end{matrix}\right)
   \left(\begin{matrix} 1 \\ 1 \end{matrix}\right)
   \, = \, \dfrac{b^{}_k}{2} \left(\begin{matrix} 1 \\ 2
   \end{matrix}\right),
\]
we have 
\begin{equation}\label{eq:ASk}
    \left(\begin{matrix} {f}\left(\frac{b^{}_k 2^k +
   \cdots + b^{}_1 2 + b^{}_0}{2^{k+1}}\right)\\ 1\end{matrix}\right)
   =\myfrac{1}{3}\left(\begin{matrix} {S}^{}_{b_k} & 
   {u}^{}_{b_k}\\ 0\ \ \ 0 & 3 \end{matrix}\right)
    \left(\begin{matrix} {f}\left(\frac{b^{}_{k-1} 2^{k-1}+
    \cdots + b^{}_1 2 + b^{}_0 }{2^{k}}\right)\\ 1\end{matrix}\right).
\end{equation}
Iterating Eq.~\eqref{eq:ASk} we get
\begin{equation}\label{eq:fSAf}
   \left(\begin{matrix} {f}\left(\frac{b^{}_k 2^k +
   \cdots + b^{}_1 2 + b^{}_0 }{2^{k+1}}\right)\\ 1\end{matrix}\right)
   =\myfrac{1}{3^k}\left(\begin{matrix} {S}^{}_{b_k} & {u}^{}_{b_k}\\
    0\ \ \ 0 & 3 \end{matrix}\right)
   \cdots\left(\begin{matrix} {S}^{}_{b_{1}} & {u}^{}_{b_{1}}\\ 
    0\ \ \ 0 & 3\end{matrix}\right)\left(\begin{matrix}
    {f}\left(\frac{b^{}_0}{2} \right)\\ 1\end{matrix}\right).
\end{equation}

For the integer $2m\in[0,2^{k})$, consider the binary expansions 
$2^k+2m+1=1b_{k-1}\cdots b_11$ and  $2^k+2m=1b_{k-1}\cdots
b_10$ in obvious notation. Then, using 
Eqs.~\eqref{eq:fdiffm} and \eqref{eq:fSAf} and
observing that $f \bigl( \frac{1}{2}\bigr) = \bigl(
\frac{1}{6}, \frac{1}{3} \bigr)^{T}$, we have
\begin{align*}
   \mu([2m/2^k, (2m+1)/2^k)]) \,
   &= \, (3^{1-k}, 0 , 0)\left(\begin{matrix} {S}^{}_{1} & 
   {u}^{}_{1}\\ 0\ \ \ 0 & 3 \end{matrix}\right)
   \cdots\left(\begin{matrix} {S}^{}_{b_{1}} & {u}^{}_{b_{1}}\\
    0\ \ \ 0 & 3 \end{matrix}\right)\left[\left(\begin{matrix}
    {f}\left(\frac{1}{2}\right)\\ 1\end{matrix}\right)-\left(
   \begin{matrix} {f}\left(0\right)\\ 1\end{matrix}\right)\right]\\
   &= \, \myfrac{1}{6}
   (3^{1-k}, 0 , 0)\left(\begin{matrix} {S}^{}_{1} & {u}^{}_{1}\\ 
   0\ \ \ 0 & 3 \end{matrix}\right)
   \cdots\left(\begin{matrix} {S}^{}_{b_{1}} & {u}^{}_{b_{1}}\\ 
   0\ \ \ 0 & 3\end{matrix}\right)\left(\begin{matrix}1 \\
    2 \\ 0\end{matrix}\right)\\
   &= \, \myfrac{1}{6}
   (3^{1-k}, 0)\, {S}^{}_{1}\, {S}^{}_{b^{}_{k-1}}\,\cdots\,
   {S}^{}_{b^{}_{1}}\left(\begin{matrix}1 \\ 2 \end{matrix}\right)
   \, \geqslant \, \myfrac{3}{6\cdot3^k} \, > \, 0 \ts ,
\end{align*} which proves the theorem.
\end{proof}

Unlike other distribution functions of singular continuous measures,
the distribution function $F$ from Eq.~\eqref{eq:distmu} looks
relatively `smooth'; this is quantified in the following corollary.

\begin{coro}\label{coro:holder}
  The distribution function $F$ from Eq.~\eqref{eq:distmu} is
  H\"{o}lder continuous with exponent $\log_2(3/\tau)$.
\end{coro}

\begin{proof} 
  The distribution function $F$ inherits this exponent from the
  dilation equation for the assiciated function $f$. One can see this
  by following through the proof of Theorem~\ref{thm:strict} up to
  Eq.~\eqref{eq:fdiffm} using the interval $[x,y]\subseteq[0,1]$.
\end{proof}

Further, it will be an interesting question to analyse some of the
scaling properties of $\mu$, for instance in analogy to the treatment
of the Thue--Morse measure in \cite{BGN}.  For this, it will be
helpful to understand the precise relation between the measure $\mu$
and the dilation equation from Proposition~\ref{prop:s}. One such
relation can be stated as follows.

\begin{coro}
  For\/ $x\in [0,1]$, the distribution function\/ $F$ from
  Eq.~\eqref{eq:distmu} satisfies
\[
   F (x) \, = \, 3 \left( f^{}_{0} \bigl( \tfrac{1+x}{2} \bigr)
   - f^{}_{0} \bigl( \tfrac{1}{2} \bigr) \right) \, = \,
   f^{}_{0} (x) + f^{}_{1} (x) \ts ,
\]
  where\/ $f^{}_{0}$ and\/ $f^{}_{1}$ are
  the functions from Proposition~\emph{\ref{prop:s}}.
\end{coro}

\begin{proof}
  The first identity is a rather direct consequence of
  Eq.~\eqref{eq:fdiffm} in the proof of Proposition~\ref{prop:s}.
  While $f^{}_{0} \bigl( \frac{1}{2}\bigr) = \frac{1}{6}$, the
  dilation equation for the functions $f_i$ gives
\[
   f^{}_{0} \bigl( \tfrac{1+x}{2} \bigr) \, = \,
   \myfrac{1}{6} + \myfrac{1}{3} \bigl( f^{}_{0} (x)
   + f^{}_{1} (x) \bigr) .
\]
  This implies the second identity.
\end{proof}

\section*{Appendix}

The purpose of this appendix is to provide two other proofs of the
continuity of the measure $\mu$ from Proposition~\ref{prop:gen}, in
view of their potential usefulness in other applications to recursive
sequences of a similar kind. \smallskip

The first one starts from the observation that
$\mu = \medconv_{m\geqslant 1} \nu^{}_{m}$, with the probability
measures
$\nu^{}_{m} = \frac{1}{3} \bigl( \delta^{}_{0} + \delta^{}_{2^{-m}} +
\delta^{}_{-2^{-m}}\bigr)$
on $\TT$ as building blocks, is absolutely convergent in the weak
topology, which is to say that that it is weakly convergent to the
same limit in any order of its terms.  This follows from
\cite[Thm.~6]{JW}, where one has to notice that
\[
  M_{r} (\nu^{}_{m}) \, := \int_{\TT} x^{r} \dd \nu^{}_{m} (x)
  \, = \, \myfrac{1}{3} \bigl( 0^r + 2^{-rm} + (-1)^{r} 2^{-rm} \bigr) 
\]
for $r\geqslant 0$. In particular, $M_1 (\nu^{}_{m}) = 0$ and
$M_2 (\nu^{}_{m}) = \frac{2}{3}\ts 4^{-m}$, so both
$\sum_{m=1}^{\infty} \lvert M_1 (\nu^{}_{m}) \rvert$ and
$\sum_{m=1}^{\infty} M_2 (\nu^{}_{m})$ are finite, which ensures
absolute convergence.

If we now assume, contrary to the claim, that $\mu$ fails to be
continuous, there must be an $x\in\TT$ with $\mu(\{ x\}) >0$.  Now,
rewrite $\mu$ as $\mu = \nu^{}_{n} * \rho^{}_{n}$ with
$\rho^{}_{n} := \medconv_{m\ne n} \nu^{}_{m}$, which is possible for
any $n\in\NN$. This implies
\[
   \mu (\{ x \}) \, = \, \bigl(\nu^{}_{n} * \rho^{}_{n}\bigr)
   (\{ x \} ) \, = \, \myfrac{1}{3} \bigl( \rho^{}_{n} (\{ x \})
   + \rho^{}_{n} (\{ x + 2^{-n}\}) + \rho^{}_{n} (\{ x - 2^{-n}\})
   \bigr) .
\]
Analogously, one obtains
\[
  \mu (\{ x \pm 2^{-n}\}) \,  = \, \myfrac{1}{3} \bigl( 
  \rho^{}_{n} (\{ x \}) + \rho^{}_{n} (\{ x \pm 2^{-n}\}) 
   + \rho^{}_{n} (\{ x \pm 2 \cdot 2^{-n}\}) \bigr).
\]
Together with the previous relation, this implies the
estimate\footnote{More generally, one has the relation
  $\nu^{}_{n} \leqslant \bigl(\delta^{}_{2^{-n}} +
  \delta^{}_{-2^{-n}}\bigr)*\nu^{}_{n}$
  as an inequality between positive measures, and hence --- by
  convolution with $\rho^{}_{n}$ --- also
  $\mu \leqslant \bigl(\delta^{}_{2^{-n}} + \delta^{}_{-2^{-n}}
  \bigr)* \mu$, which implies Eq.~\eqref{eq:JW-trick}. }
\begin{equation}\label{eq:JW-trick}
   \mu(\{x\}) \, \leqslant \, \mu (\{ x + 2^{-n}\})
   + \mu (\{ x - 2^{-n}\}) \ts .
\end{equation}

Now, choose $r\in\NN$ with $\mu(\{x\}) > \frac{1}{r}$, and select
$r$ integers $j^{}_{1} < j^{}_{2} < \ldots < j^{}_{r}$ with $j^{}_{1}
\geqslant 2$. Since $\mu$ is a probability measure on $\TT$, we
then get
\[
\begin{split}
  1 \, & \geqslant \, \mu  \biggl( 
   \dot{\bigcup_{1\leqslant  q\leqslant  r}} \bigl(
     \big\{x-2^{- j_q} \big\}\, \dot{\cup}\,
     \big\{x+2^{- j_q} \big\}\bigr) \biggr) \\
    & = \, \sum_{q=1}^r \Bigl( \mu \bigl(  
     \big\{x-2^{- j_q} \big\}\bigr) +
     \mu \bigl( \big\{x+2^{- j_q} \big\} \bigr)  
     \Bigr)  \,  > \, r \ts \mu \bigl(\{x\}\bigr) 
    \; \geqslant \: r \, \myfrac{1}{r} \; = \; 1 \ts .
\end{split}
\]
This contradiction shows that $\mu$ is continuous. \smallskip

The second alternative proof employs Wiener's criterion
again. Observing (without proof) that the inequalities
\begin{equation}\label{eq:more}
  \lvert \widehat{\mu} (2k+1)\rvert \, \leqslant \, 
  \frac{1}{2} \lvert \widehat{\mu} (k) + \widehat{\mu} 
  (k+1)\rvert 
  \quad \text{and} \quad
  \widehat{\mu} (2k+1) \bigl( \widehat{\mu} (2k) 
  + \widehat{\mu} (2k+2)\bigr) \, \leqslant \, 0
\end{equation}
hold for all $k\in\ZZ$, one can proceed as follows. 
With $\varSigma(N) := \sum_{k=-N}^{N} \widehat{\mu} (k)^2$, 
one has
\[
\begin{split}
  \varSigma(4N) \, & = \sum_{k=-2N}^{2N} \widehat{\mu} (2k)^2
   \, + \! \sum_{k=-2N}^{2N-1} \widehat{\mu} (2k+1)^2 
      \, \leqslant \, \varSigma(2N) \, + \,
    \myfrac{1}{4} \sum_{k=-2N}^{2N-1} \bigl( 
    \widehat{\mu} (k) + \widehat{\mu} (k+1)\bigr)^2 \\
  & = \, \myfrac{3}{2} \varSigma(2N)  \, - \, 
    \myfrac{\widehat{\mu} (2N)^2}{2}
    \, + \, \myfrac{1}{2} \sum_{k=-2N}^{2N-1} \widehat{\mu} (k)
     \, \widehat{\mu} (k+1) \\
  & \leqslant \,  \myfrac{3}{2} \varSigma(2N) \, + \,
    \myfrac{1}{2} \sum_{k=-N}^{N-1} \widehat{\mu} (2k+1) 
    \bigl( \widehat{\mu} (2k) + \widehat{\mu} (2k+2)\bigr)
   \, \leqslant \, \myfrac{3}{2} \varSigma(2N) \ts ,
\end{split}
\]
where Lemma~\ref{lem:FB} was used several times,
while Eq.~\eqref{eq:more} was needed for the first
and the last inequality.

This estimate implies
$\varSigma(2^{k+1})\leqslant \bigl(\frac{3}{2}\bigr)^{k}\varSigma(2)$
and thus also
$\varSigma(N) \leqslant C \ts ( \frac{3}{2} )_{}^{\log_{2} (N)}$ for
some positive constant $C$. With $\alpha = \log_{2} (3/2) <1$, one
then obtains the asymptotic behaviour
$\frac{1}{N} \varSigma(N) = O(1/N^{1-\alpha})$ as $N\to\infty$, which
implies the absence of pure point components in $\mu$ by Wiener's
criterion. This approach has the advantage (over the Jessen--Wintner
argument) that one also gets a lower bound on the H\"{o}lder exponent
from Corollary~\ref{coro:holder}

\section*{Acknowledgements}
It is our pleasure to thank Uwe Grimm, Jeff Hogan, Bj\"{o}rn
R\"{u}ffer, Timo Spindeler and Wadim Zudilin for helpful discussions.
This work was supported by the German Research Council (DFG), within
the CRC 701, and by the Australian Research Council, through grant
DE{\ts}140100223.

\end{document}